\documentclass[11pt,reqno]{amsart}
\usepackage{amssymb}
\usepackage{amscd}
\usepackage{amsfonts}
\usepackage{mathrsfs}
\usepackage{setspace}
\usepackage{version}
\usepackage{mathtools}
\usepackage{amsmath}
\usepackage{tikz}
\usepackage[pdftex,colorlinks,citecolor=blue]{hyperref}

\usepackage{float}
\usepackage{tikz-cd}

\usepackage{tikz}
\usetikzlibrary{calc}
\usetikzlibrary{shadings,intersections}
\usetikzlibrary{decorations.text}
\usepackage{pgfplots,wrapfig}
\usepackage{lipsum}

\usepackage{xcolor}

\theoremstyle{plain}
\numberwithin{equation}{section} 
\newtheorem{theorem}[subsection]{Theorem}
\newtheorem{proposition}[subsection]{Proposition}
\newtheorem{lemma}[subsection]{Lemma}

\newtheorem{claim}[subsection]{Claim}
\theoremstyle{remark}
\newtheorem{remark}[subsection]{Remark}

\usepackage{soul}

\renewcommand{\leq}{\leqslant}
\renewcommand{\geq}{\geqslant}
\newsavebox{\proofbox}
\savebox{\proofbox}{\begin{picture}(7,7)  \put(0,0){\framebox(7,7){}}\end{picture}}

\newcommand\E{\mathbb{E}}

\newcommand\R{\mathbb{R}}
\newcommand\PV{\textrm{P}(V)}
\newcommand\p{\mathbb{P}}
\newcommand\C{\mathbb{C}}
\newcommand\N{\mathbb{N}}

\newcommand\SL{\operatorname{SL}}

\newcommand\GL{\operatorname{GL}}

\newcommand\End{\operatorname{End}}

\newcommand\diag{\operatorname{diag}}

\newcommand{\efface}[1]{}

\begin{document}

\title[Law of large numbers for the spectral radius]{Law of large numbers for the spectral radius of random matrix products}

\author{Richard Aoun}
\address{American University of Beirut, Department of Mathematics, Faculty of Arts and Sciences, P.O. Box 11-0236 Riad El Solh, Beirut 1107 2020, Lebanon}
\email{ra279@aub.edu.lb}
\author{Cagri Sert}
\address{\textsc{Departement Mathematik, ETH Z\"{u}rich, 101 R\"{a}mistrasse, 8092, Z\"{u}rich, Switzerland}}
\email{sertcagri@gmail.com}


\begin{abstract}
We prove that the spectral radius of an i.i.d.\ random walk on $\GL_d(\C)$ satisfies a strong law of large numbers under finite second moment assumption and a weak law of large numbers under finite first moment. No irreducibility assumption is supposed.
\end{abstract} 
\maketitle

\section{Introduction}

Let $\mu$ be a probability measure on $\GL_d(\mathbb{C})$. For $n \in \mathbb{N}$, denote by $\mu^{\ast n}$ the $n$-fold convolution of $\mu$ by itself. This is the distribution of the product $X_n \cdots X_1$, where $X_i$'s are independent and identically distributed (i.i.d.) random variables with distribution $\mu$. Recall that a probability measure $\mu$ on $\GL_d(\mathbb{C})$ is said to have finite $k^{th}$ moment if 
$$
\int (\log N(g))^k d\mu(g) < \infty,
$$
where $N(g)=\max \{||g||, ||g^{-1}||\}$ for some operator norm $||.||$. This definition does not depend on the choice of the operator norm. 

The classical law of large numbers of Furstenberg-Kesten \cite{furstenberg-kesten} describes the asymptotic growth rate of the product $L_n:=X_n \cdots X_1$ \textit{in norm}. More precisely, it says that if $\mu$ has finite first moment, then
$$
\frac{1}{n} \log ||X_n \cdots X_1|| \underset{n \to \infty}{\overset{a.s.\ }{\longrightarrow}} \lambda_1(\mu),
$$
where the constant $\lambda_1(\mu) \in \mathbb{R}$ is defined by this almost sure convergence and called the first Lyapunov exponent of $\mu$. Nowadays, this convergence can be deduced from the subadditive ergodic theorem due to Kingman \cite{kingman}. 

The aim of this paper is to establish the analogous results for the \emph{spectral radius}. Unlike the operator norm, the behaviour of the spectral radius is highly chaotic under multiplication of matrices. For instance, it does not enjoy a simple property such  as submultiplicativity. 

Our first result is the following law of large numbers for the spectral radius under finite second moment assumption. For  $g\in \GL_d(\C)$, let $\rho(g)$ denote its spectral radius. 

\pagebreak
\begin{theorem}[Strong law of large numbers]\label{main}
Let $\mu$ be a probability measure on  $\GL_d(\mathbb{C})$ with finite second moment and $(X_n)_{n\in \N}$ a sequence of independent random variables with distribution $\mu$. For every $n\in \N$, let $L_n=X_n\cdots X_1$. Then
$$
\frac{1}{n} \log \rho(L_n) \underset{n \to \infty}{\overset{a.s.\ }{\longrightarrow}} \lambda_1(\mu). 
$$
\end{theorem}

Our second result is the weak law of large numbers for the spectral radius under the optimal assumption of finite first moment.

\begin{theorem}[Weak law of large numbers]\label{main1}
Let $\mu$ be a probability measure on  $\GL_d(\mathbb{C})$ with finite first moment. With the notation of Theorem \ref{main},  $\frac{1}{n} \log \rho(L_n)$ converges to $\lambda_1(\mu)$ in $\textrm{L}^1$, i.e.  
$$\E \left( \big| \frac{1}{n}\log\rho(L_n) - \lambda_1(\mu) \big| \right) \underset{n\to+\infty}{\longrightarrow} 0.$$
In particular,   $\frac{1}{n}\log\rho(L_n)$ converges to $\lambda_1(\mu)$ in  probability. 
\end{theorem}

\begin{remark}[Convergence of the moduli of other eigenvalues]
More generally, our results yield the corresponding laws of large numbers for the moduli of \textit{all} eigenvalues of the random matrix products: given $g \in \GL_d(\mathbb{C})$, denote by $\vec{\rho}(g)=\left(\rho_1(g), \ldots, \rho_d(g)\right)$ 
the $d$-vector of moduli of eigenvalues of $g$ ordered by decreasing order. Then, if $\mu$ has finite second moment the following  convergence holds almost surely 
$$\frac{1}{n} \log \vec{\rho}(L_n) \underset{n\to +\infty}{\longrightarrow} \vec{\lambda}(\mu),$$ where $\vec{\lambda}(\mu)=\left( \lambda_1(\mu), \cdots, \lambda_d(\mu) \right)$ is the Lyapunov vector of $\mu$ (see \S \ref{sub.RMP}). The analog result for the weak law is also true. 
\end{remark}

A strong law for the spectral radius was previously proved by Guivarc'h \cite[Th\'eor\`eme 8]{Guivarch3} and Benoist--Quint \cite[Theorem 14.12]{BQ.book} under two additional assumptions. The first one is of algebraic nature and supposes that the Zariski closure of the  group $G_\mu$ generated by the support of $\mu$ acts strongly irreducibly on $\mathbb{C}^d$. The second one is a moment hypothesis and supposes that $\mu$ has finite exponential moment. 

Our approach is purely dynamical: it is based on the first Lyapunov gap, i.e. the integer $s=1, \ldots, d$, such that $\lambda_1(\mu)=\cdots = \lambda_s(\mu)>\lambda_{s+1}(\mu)$. 
This allows us to obtain Theorems \ref{main} and \ref{main1} without the strong irreducibility assumption. Such an approach is inspired by \cite{Aoun-Guivarch}. The improvement concerning the moment assumption on $\mu$ is based on large deviation estimates for cocycles established by Benoist--Quint in \cite{BQ.CLT}.

A notable feature of the laws of large numbers for the spectral radius concerns the i.i.d.\ assumption. The Furstenberg--Kesten law of large numbers for the norm holds for any 
random product $L_n=X_n \cdots X_1$ with ergodic stationary increments $X_i$. In Section \ref{end.proof}, we present a simple example of an ergodic stationary random walk with Markovian increments for which Theorems \ref{main1} and \ref{main} fail. Similar examples can be found in the work of Avila--Bochi \cite{avila-bochi}. On the other hand, we recall that for any ergodic stationary random walk $L_n$, 
$$\underset{n\rightarrow +\infty}{\limsup}\frac{1}{n}\log \rho(L_n) \overset{\text{a.s.\ } }{=}\lambda_1(\mu),$$ thanks to Morris \cite{morris} (see also \cite{avila-bochi} and \cite{oregonreyes}).

Our proof of the strong law of large numbers is based on a careful quantitative analysis of the distance between the attractive point and the repelling hyperplane  in the projective space of random products $L_n$. This approach based on, among others, large deviation estimates is the reason why we are only able to prove the strong law of large numbers under finite second moment assumption. The corresponding question with finite first moment remains open.  
On the other hand, the proof of the weak law needs only a qualitative version of the proof of the strong law. This accounts for the weaker and optimal finite first moment assumption in Theorem \ref{main}.

Finally we note that both of Theorems \ref{main1} and \ref{main} are also valid for the right random walk $R_n=X_1 \cdots X_n$ under the same assumptions, by considering the  transposed random walk.  Also, Theorems \ref{main} and \ref{main1} remain true if we replace $\C$ by any local field of arbitrary characteristic. Since the proofs can be adapted   without any substantial difficulty, we will restrict ourselves to $\C$.

\subsection*{Acknowledgments} Both authors have the pleasure to thank AUB, CAMS and ETH  Z\"{u}rich for hospitality and excellent working conditions. C.S.\  is supported by SNF grant 178958.

\section{Preliminaries}

\subsection{Linear algebra}
We write $V=\mathbb{C}^d$; we endow $V$ with the canonical Hermitian structure and denote by $||.||$ the associated norm on $V$   as well as the operator norm on $\End(V)$. For convenience, we shall work with this choice norm but we remind that our results and hypotheses do not depend on such a choice of norm. We denote by $\bigwedge^k V$ the $k^{th}$-exterior power of $V$ that we endow with the induced Hermitian structure.

Let $\textrm{P}(V)$ denote the projective space of $V$. For every non-zero vector $v \in V$ (resp.\ non-trivial subspace $E$ of $V$), we denote by $[v]=\C v$ (resp.~$[E]$) its image in $\textrm{P}(V)$. We  endow $\textrm{P}(V)$ with the Fubini-Study metric $\delta$ defined by: 
$$\forall \, x=[v], y=[w]\in \textrm{P}(V), \delta(x,y):=\frac{|| v \wedge w||}{||v||\,||w||}.$$
The action of $g\in \textrm{GL}(V)$ on a vector $v$ will be simply denoted by $g v$, while the action of $g$ on a point $x \in \textrm{P}(V)$ will be denoted by $g\cdot x$. 

Let $(e_1, \ldots, e_d)$ be the canonical basis of $V$. Denote by $K=U_d(\C)$ the unitary group, by $A\subset \textrm{GL}_d(\C)$ the subgroup of diagonal matrices with positive real coefficients. We denote by $A^+ \subset A$ the sub-semigroup given by $A^+:=\{\diag(a_1,\ldots,a_d) \, | \, a_1 \geq a_2 \geq  \cdots \geq a_d \}$. The KAK decomposition (or the singular value decomposition) states that  $\textrm{GL}_d(\C)=K A^+ K$. For an element $g\in \textrm{GL}_d(\C)$, we denote by $g=k_g a(g) l_g$ a KAK decomposition of $g$, i.e. $k_g,l_g \in K$ and $a(g)=\diag(a_1(g), \ldots, a_d(g)) \in A^+$. In this decomposition, the middle factor $a(g)$ is uniquely defined. Moreover, $$\frac{a_2(g)}{a_1(g)}=\frac{||\bigwedge ^2 g||}{||g||^2}.$$
Even though the choice of $k_g$ and $l_g$ is not unique, we fix once for all a privileged choice of a KAK decomposition. We call \emph{attracting point} and \emph{repelling hyperplane}, respectively, the following point in $\textrm{P}(V)$ and projective hyperplane of $\textrm{P}(V)$: 

$$x_g^+=[k_g e_1] \,\,\,,\,\,\,  H_g^{<} = [\langle l_g^{-1} e_2, \cdots, l_g^{-1} e_d \rangle].$$
Observe that $$H_g^<=(\C\, x_{g^\ast}^+)^\perp.$$

We finally recall the following lemma that allows one to compare the norm and the spectral radius of a matrix using the geometry of its projective action. 

\begin{lemma}\cite[Lemma 14.14]{BQ.book}
Let $g\in \textrm{GL}(V)$.
If $\delta(x_g^+, H_g^{<})> 2 \sqrt{\frac{a_2(g)}{a_1(g)}}$, then 
$$  \frac{\rho(g)}{||g||} \geq   \frac{\delta(x_g^+, H_g^{<})}{2}.$$
\label{geometry}\end{lemma}

\subsection{Some background on random matrix products}\label{sub.RMP} In this part, we recall some classical results of random matrix products theory, and prove Proposition \ref{largedeviations1} and Proposition \ref{large.deviations.2}. The former is based on large deviations for cocycles due to Benoist--Quint (Theorem  \ref{BQlarge}) and will be key to prove the strong law with \textit{a finite second moment}.  The latter  will be useful  for the weak law  under \textit{a finite first moment}. It relies on Lemma \ref{lemma.ps} which generalizes a known fact  under strong irreducibility assumption. 

\bigskip

Let $\mu$ be a probability measure on $\GL(V)$ with finite first moment. We denote by   $G_{\mu}$ the  group generated by its support. Let $L_\mu$ be the subspace of $V$ given by
$$L_{\mu}=\{v\in V \setminus \{0\} \, | \; \underset{n\rightarrow +\infty}{\limsup} \frac{1}{n} \log ||L_n v ||<\lambda_1(\mu) \; \textrm{a.s.}\}.$$
Furstenberg--Kifer \cite[Theorem 3.9]{furstenberg-kifer} and  Hennion \cite[Th\'{e}or\`{e}me 1]{hennion} showed that $L_{\mu}$ is the largest $G_{\mu}$-invariant subspace of $V$ such that the Lyapunov exponent of the restriction of $L_n$ is strictly smaller than $\lambda_1(\mu)$. Moreover, for every $v\in V\setminus \{0\}$,  the $\limsup$ in the definition is actually an almost sure limit. 
 
A probability measure $\nu$ on $\PV$ is said to be $\mu$-stationary if $\mu \ast \nu = \nu$, in other words, if for every continuous function $f$ on $\PV$, $\iint {f(g\cdot x)\,d\mu(g)\,d\nu(x)}=\int{f \,d\nu}$. Furstenberg--Kifer \cite{furstenberg-kifer} and Hennion \cite{hennion} showed that $\lambda_1(\mu)$ can be expressed as \begin{equation}\label{double}\lambda_1(\mu)=\sup\left\{\iint \log\frac{||g v||}{||v||}\,d\mu(g) d\nu([v]) \,  \bigg|   \, \mu \ast \nu = \nu\right\}.\end{equation}
In particular, when $L_{\mu}=\{0\}$, for every $\mu$-stationary probability measure $\nu$ on $\PV$, 
$\iint{\log\frac{||g v||}{||v||}\,d\mu(g)\,d\nu([v])}=\lambda_1(\mu)$, i.e.\ the additive cocycle $\sigma(g,[v]):=\log \frac{||g v||}{||v||}$ has unique cocycle average. Here, by an additive cocycle, we mean a function $\sigma$ on $G \times X$ that satisfies $\sigma(gh,x)=\sigma(g,h\cdot x)+\sigma(h,x)$ for every $(g,h)\in G^2$ and $x\in X$. Furthermore, if $\pi: G_{\mu} \longrightarrow \GL(V/L_\mu)$ is the canonical projection and $\pi_\ast \mu$ is the pushforward of $\mu$ by $\pi$, then since the Lyapunov exponent of the restriction of $L_n$ to $L_\mu$ is strictly smaller than $\lambda_1(\mu)$, it follows from \cite[Lemma 3.6]{furstenberg-kifer} that we have $\lambda_1(\pi_\ast \mu)=\lambda_1(\mu)$  and $L_{\pi_\ast \mu}=\{0\}$.

Below, we recall the following 

\begin{theorem}
\cite[Proposition 3.2]{BQ.CLT}\label{BQlarge}
Let $G$ be a locally compact group, $X$ a compact metrizable G-space, $\mu$ a   probability measure on $G$. Let $\sigma: G\times X \longrightarrow \R$  be a cocycle such that 
$\int_{G}{\left(\sup_{x\in X}{|\sigma(g, x)|}\right)^2 d\mu(g)}<+\infty$.  Let  $\sigma_\mu^+$ and $\sigma_\mu^-$ be its    upper and lower average, i.e. 
$$\sigma_\mu^+ =\sup\left\{\iint{\sigma(g,x)\,d\mu(g) d\nu(x) \big| \mu \ast  \nu=\nu}\right\}$$ and   similarly for $\sigma_{\mu}^-$ with $\sup$ replaced by $\inf$.  Then, for any $\epsilon>0$, there exists a sequence $D_n$ of positive reals with $\sum_nD_n<+\infty$ such that for every $n\in \N$ and $x\in X$, $$\p\left(\frac{\sigma(L_n, x)}{n} \in [\sigma_{\mu}^- -\epsilon, \sigma_\mu^+ +\epsilon]\right)>1-D_n.$$
\end{theorem}

We note that similar large deviation estimates under weaker moment assumptions were obtained by Cuny--Dedecker--Merlev\`{e}de in \cite{cuny}.

The following result will be crucial to our considerations. 

\begin{proposition}\label{largedeviations1} Let $\mu$ be a probability measure on $\GL(V)$. Assume that $\mu$ has finite second moment. Then for every $\epsilon>0$, there exists a sequence $(D_n)_{n\in \N}$ of positive reals with $\sum_n{D_n}<+\infty$ such that for every $n\in \N$ and  $v\in V\setminus \{0\}$, 
\begin{equation}\label{y1}\p \left( \delta([v], [L_{\mu}])\, e^{n(\lambda_1(\mu) - \epsilon)} \leq \frac{||L_n v||}{||v||} \leq e^{n(\lambda_1(\mu) + \epsilon)}\right) >1-D_n,\end{equation}
and 
\begin{equation}\label{y2}\p \left(e^{n(\lambda_1(\mu) - \epsilon)} \leq  ||L_n|| \leq e^{n(\lambda_1(\mu) + \epsilon)}\right) >1-D_n.\end{equation}
\end{proposition}

This proposition is precisely \cite[Proposition 4.1]{BQ.CLT} except for \eqref{y1} where the authors assume $G_{\mu}$   to be irreducible (in particular,  $L_{\mu}=\{0\}$). However, their proof gives, without further substantial difficulty, the slightly more precise estimate \eqref{y1}. For the convenience of the reader, we include a proof below. 

\begin{proof} It is enough to prove \eqref{y1} since \eqref{y2} follows from \eqref{y1} by applying it to each vector of a fixed orthonormal basis of $V$. 
Let $L=L_{\mu}$ and $G=G_{\mu}$ for simplicity of notation. Also, let $\sigma: G \times \PV \longrightarrow \R, (g, [v])\longmapsto \log \frac{||g v||}{||v||}$. It is readily seen that $\sigma$ is an additive cocycle.  By \eqref{double},  its upper cocycle average is   $\sigma_{\mu}^+=\lambda_1(\mu)$. Thus, by Theorem  \ref{BQlarge}, it is enough to prove  that for every $\epsilon>0$,  there exists a summable sequence $(D_n)_n$ such that for every $v\in V\setminus L$ and  $n
\in \N$,  
$$\p \left( \delta([v], [L_{\mu}])\, e^{n(\lambda_1(\mu) - \epsilon)} \leq \frac{||L_n v||}{||v||}\right) >1-D_n.$$

For $v\in V\setminus L$, denote by $\overline{v}$ its projection onto the quotient $V/L$. Endow $V/L$ with the norm $||v+L||:=\inf\{||v+w||; w\in L\}$. Denote by $\pi: G\longrightarrow \GL(V/L)$ the canonical projection. Then, for every $g\in G$ and\ $v\in V \setminus L$, $$\delta([v], [L])=\frac{||\overline{v}||}{||v||},$$
 and 
\begin{equation}\label{eq.norms}
\frac{||g v||}{||v||}\geq \frac{||\overline{gv}||}{||v||}=  \delta([v], [L])  \frac{||\pi(g)\overline{v}||}{||\overline{v}||}.
\end{equation}
Moreover, the probability measure $\pi_\ast \mu$ on $V/L$ has unique cocycle average $\sigma_\mu^+=\sigma_\mu^-=\lambda_1(\pi_\ast \mu)=\lambda_1(\mu)$. 
Applying again Theorem \ref{BQlarge}, we obtain the desired estimate. 
\end{proof}

Now, we  introduce the tools needed for the weak law of large numbers. First, we recall that the other Lyapunov exponents $\lambda_k(\mu)$ for $k=1,\ldots,d$ alluded to in the introduction are defined  by the following almost sure limit 
\begin{equation}\label{alllyap}
\frac{1}{n}  \log  a_k(L_n)   \underset{n \to \infty}{\longrightarrow}  \lambda_k(\mu),
\end{equation}
which exists in view of Furstenberg-Kesten's theorem.

The following lemma is of independent interest and  generalizes   \cite[Corollary 3.4]{bougerol} where the result is proved under a strong irreducibility assumption. The proof combines Oseledets' theorem and Furstenberg--Kifer and Hennion's results. For a related result, see \cite[Proposition 2.3]{prohaska-sert}.
\begin{lemma}\label{lemma.ps}
Let $\mu$ be a probability measure on $\GL(V)$ with finite first moment. For any sequence $([v_n])_{n\in \N}$ in $\PV$ that converges to some $[v]\in \PV\setminus [L_{\mu}]$, we have
\begin{equation}\label{ppp}\frac{1}{n}  \log \frac{ ||L_n v_n||}{||v_n||}\underset{n\rightarrow +\infty}{\overset{\text{a.s.}}{\longrightarrow}} \lambda_1(\mu). \end{equation}
\end{lemma}

\begin{proof}
Without loss of generality, we take $||v_n||=1$. In case $\lambda_1(\mu)=\cdots = \lambda_d(\mu)$, since for every $v \in V$ with $||v||=1$, we have $||L_n v|| \geq a_d(L_n)$, it follows by \eqref{alllyap} that almost surely $\liminf_n\frac{1}{n} \log ||L_n v_n||\geq \lambda_d(\mu)=\lambda_1(\mu)$, showing \eqref{ppp}. So suppose this is not the case and let $s \in \{1,\cdots, d-1\}$ be minimal such that $\lambda_{s}(\mu)>\lambda_{s+1}(\mu)$.
 Write $L_n=k_n a(n) l_n$ in the KAK decomposition where we denote $a(n)=:\diag \left(a_{1}(n), \cdots, a_{d}(n)\right)$. Let 
$F_n^{<}$ be the subspace of $V$ given by: $$F_n^<:=\langle l_n^{-1} e_{s+1}, \cdots, l_n^{-1} e_d \rangle.$$
By Oseledets' theorem \cite{oseledets}, $F_n^{<}$ converges almost surely to a random $(d-s)$-dimensional subspace $F^<$ of $V$. For $\p$-almost every $\omega\in \Omega$, $F^<(\omega)$ satisfies: 
$$ \lim_n \frac{1}{n} \log ||L_n(\omega) u||=\lambda_1(\mu) \underset{n\to +\infty}{\Longleftrightarrow} u\not\in F^<(\omega).$$
But since $v\not\in L_{\mu}$, we deduce by Furstenberg--Kifer and Hennion  that  $\frac{1}{n} \log ||L_n v||$ converges almost surely to  $\lambda_1(\mu)$. Hence, almost surely, $v\not\in F^{<}$. Up to choosing a $\p$-full measure subset of $\Omega$, we can suppose that for every $\omega \in \Omega$, $v\not\in F^{<}(\omega)$. Now fix $\omega \in \Omega$. Since   $v_n\longrightarrow v$, there exist 
$N_0=N_0(\omega)\in \N$ and $\epsilon_0=\epsilon_0(\omega)>0$ such that for  every  $n\geq N_0$,  
$$\delta\left([v_n], F_n^<(\omega)\right)\geq \epsilon_0.$$
 Let  $n\geq N_0$ and decompose     $v_n$ in $V=F_n^<(\omega) \oplus {F_n^<}^{\perp}(\omega)$. Writing 
$v_n={v_n^<}+v_n^+$ we have     $||v_n^+||\geq \epsilon_0$ and 
we obtain 
\begin{eqnarray} 
||L_n(\omega) v_n||&=& ||L_n(\omega) v_n^+ + L_n(\omega) v_n^<||\nonumber\\
& \geq & ||L_n(\omega) v_n^+|| - ||L_n(\omega) v_n^<||\nonumber\\
&\geq &  a_s(n)\epsilon_0  - a_{s+1}(n).\nonumber
\end{eqnarray}
By the defining property of $s \in \mathbb{N}$, we deduce by \eqref{alllyap} that 
$$\underset{n\to +\infty}{\liminf} \frac{1}{n}\log||L_n(\omega) v_n||\geq \lambda_1(\mu).$$
This   finishes the proof. 
\end{proof}

We deduce a qualitative version of Proposition \ref{largedeviations1} under a finite first moment assumption:

\begin{proposition}\label{large.deviations.2}
Assume that  $\mu$ has finite first moment. Then for every $\epsilon>0$, 

\begin{equation}\label{p1}\p \left( \delta([v], [L_{\mu}])\, e^{n(\lambda_1(\mu) - \epsilon)} \leq \frac{||L_n v||}{||v||} \leq e^{n(\lambda_1(\mu) + \epsilon)}\right) \underset{n\to+\infty}{\longrightarrow} 1,\end{equation}
uniformly in $[v] \in \PV$, and 
\begin{equation}\label{p2}\p \left(e^{n(\lambda_1(\mu) - \epsilon)} \leq  ||L_n|| \leq e^{n(\lambda_1(\mu) + \epsilon)}\right) \underset{n\to+\infty}{\longrightarrow} 1.\end{equation}
\end{proposition}

\begin{proof}
Convergence \eqref{p2} follows immediately from Furstenberg--Kesten's theorem. We prove \eqref{p1}. Let $\epsilon>0$. 
Using \eqref{p2}, it is enough to show that 
$\p \left( \delta([v], [L_{\mu}])\, e^{n(\lambda_1(\mu) - \epsilon)} \leq \frac{||L_n v||}{||v||}\right) \underset{n\to+\infty}{\longrightarrow} 1$ uniformly 
in $v\in V\setminus L_\mu$. Moreover, by compactness of $P(V/L_{\mu})$,  it is enough to show that for every sequence $(v_n)_{n\in \N}$ in $V\setminus L_{\mu}$ such that $([\overline{v}_n])_{n\in \N}$   converges in $P(V/L_{\mu})$, we have
\begin{equation}\label{i8}\p \left( \delta([v_n], [L_{\mu}])\, e^{n(\lambda_1(\mu) - \epsilon)} \leq \frac{||L_n v_n||}{||v_n||} \right) \underset{n\to+\infty}{\longrightarrow} 1.
\end{equation}
Recall that by  \eqref{eq.norms} for every $v \in V \setminus \{0\}$ we have
\begin{equation}\label{eqq}\frac{||L_n v||}{||v||}\geq \frac{||\pi(L_n) \overline{v}||}{||\overline{v}||} \delta([v], [L_{\mu}]),
\end{equation}
where $\pi: G_{\mu} \longrightarrow \GL(V/L_{\mu})$ is the canonical projection. 
Now, since $L_{\pi_\ast \mu}=\{0\}$ and $\lambda_1(\pi_\ast \mu)=\lambda_1(\mu)$, the convergence \eqref{i8} follows directly from Lemma \ref{lemma.ps}. 
\end{proof}

\section{Estimates when $\lambda_1(\mu)>\lambda_2(\mu)$}
In this section, we give estimates for the random matrix products when $\lambda_1(\mu)>\lambda_2(\mu)$. Our main goal is Proposition \ref{prop.distance} which shows that with high probability the attractive point $x_{L_n}^+$ of the random walk $L_n$ does not stay exponentially close to the repelling hyperplane $H_{L_n}^<$. 
Proposition \ref{prop.distance} treats the $\textrm{L}^2$ case, while its qualitative version, Proposition \ref{prop.distance.1}, treats the $\textrm{L}^1$ case. 
\bigskip

Given a probability measure $\mu$ on $\GL_d(\C)$, we denote by $\check{\mu}$ the pushforward of $\mu$ by the map $g\mapsto g^\ast$, where $g^\ast$ is the conjugate transpose matrix of $g$.   

\begin{proposition}\label{prop.distance}
 Assume that $\mu$ has finite second moment and that  $\lambda_1(\mu)>\lambda_2(\mu)$. 
Then,   for any $\epsilon>0$, there exists  a sequence  $(D_n)_{n\in \N}$   with $\sum_n{D_n}<+\infty$ 
   such that for every $n\in \N$,
     $$\p\left(\delta(x_{L_n}^+,H_{L_n}^<)\leq e^{-\epsilon n}\right)<  D_n.$$
\end{proposition}

The proof of  Proposition \ref{prop.distance} will be done through a series of lemmas.  
\begin{lemma}
\label{lemlem1}
Under the assumptions of Proposition \ref{prop.distance}, there exists a constant $C>0$ such that for every $x\in \PV \setminus [L_\mu]$ there exists a summable sequence $(D_n)_n$ of positive reals such that for every $n\in \N$, 
\begin{equation}\p\left( \delta(x_{L_n}^+, L_n \cdot x)\geq  e^{-C n} \right)<  D_n. \nonumber\end{equation}
\end{lemma}
\begin{proof} Let   $x=[v]\in \textrm{P}(V)\setminus [L_{\mu}]$. Given    $g\in \textrm{GL}(V)$, one can easily show that
\begin{equation}\label{eq.marre}
\delta(x_{g}^+, g \cdot x)= \delta(e_1, a_g l_g \cdot x)\leq  \frac{a_2(g)}{a_1(g)}  \frac{||g||\,||v||}{||g v||}.
\end{equation}
 Applying Proposition \ref{largedeviations1} to $||L_n||$ and $||\bigwedge^2 L_n||$ gives some summable sequence $(D'_n)_n$ such  that for every $n$, 
\begin{equation}\label{nulaij}\p\left( \frac{a_2(L_n)}{a_1(L_n)} \geq e^{-n(\lambda_1-\lambda_2)/2}\right)< D'_n.
\end{equation}
Let $\epsilon:=(\lambda_1-\lambda_2)/4>0$. Since $v\not\in L_{\mu}$, applying again Proposition \ref{largedeviations1} to $L_n$ and $L_n v$, we get  a summable sequence $(D''_n)_n$ such that for every $n$, $$\p\left(\frac{||L_n||\,||v||}{||L_n v||} \geq e^{\epsilon n}\right)<D''_n.$$
Now using \eqref{eq.marre}, the desired estimate follows by taking $C:=(\lambda_1-\lambda_2)/4$ and $D_n:=D'_n+D''_n$.\end{proof}

We deduce the following estimate from the previous lemma and Proposition \ref{largedeviations1}. It says that the attracting directions of right random products stabilize exponentially fast with high probability. It is reminiscent of the estimates appearing in the usual proofs of Oseledets' theorem (see e.g.\ \cite[\S 1]{ruelle}). 

\begin{lemma}\label{lemlem2}
Under the assumptions of Proposition \ref{prop.distance}, there exist a constant $C>0$, a summable sequence $(D_n)_n$ of positive reals such that for every $n\in \N$, 
$$\mathbb{P}\left(\delta(x^{+}_{R_{2n}},x^+_{R_n}) \geq e^{-Cn}\right) <  D_n.$$
\end{lemma}

\begin{remark}
Unlike the other lemmas in this section  as well as Proposition \ref{prop.distance}, the result of Lemma \ref{lemlem2} fails if we replace the right random walk $R_n$ by the left random walk $L_n$.
\end{remark}

\begin{proof}
  Fix $x=[v]\in \PV\setminus [L_{\mu}]$.  Clearly, for every $C>0$ and  $n\in \N$,  
\begin{equation*}\label{part}
\begin{aligned}
& \mathbb{P}\left(\delta(x_{R_{2n}}^+,x_{R_n}^+)\geq e^{-Cn}\right)   \leq    \mathbb{P}\left(\delta(R_{2n}\cdot x,x^+_{R_{2n}})\geq \frac{e^{-Cn}}{3}\right) \\ & +\mathbb{P}\left(\delta(R_{n}\cdot x,x^+_{R_{n}}) \geq \frac{e^{-Cn}}{3}\right) + \mathbb{P}\left(\delta(R_{2n}\cdot x,R_n \cdot x) \geq \frac{e^{-Cn}}{3}\right).
\end{aligned}
\end{equation*}
Since $v\not\in L_{\mu}$,  Lemma \ref{lemlem1}  gives some $C_1>0$ and some summable sequence $(D'_n)_n$ such that both  $\mathbb{P}\left(\delta(R_{n} \cdot x,x^+_{R_{n}})\geq e^{-C_1n}  \right)$ and $\mathbb{P}\left(\delta(R_{2n} \cdot x,x^+_{R_{2n}})\geq e^{-C_1 n}  \right)$
are   $<D'_n$ for any $n$. 
Thus, by \eqref{part}, it suffices to show that there exists  some $C_2>0$ and some summable  sequence $(D''_n)_n$  such that for every $n$, 
\begin{equation}\label{eq1}
\mathbb{P}\left(\delta(R_{2n} \cdot x,R_n \cdot x) \geq \frac{e^{-C_2n}}{3}\right) \leq D''_n.
\end{equation}
Indeed, the desired estimate then follows by taking $C=\min\{C_1,C_2\}/2$ and $D_n=2D'_n+D''_n$. 
To prove \eqref{eq1}, we observe that
\begin{equation*}
\delta(R_{2n} \cdot x, R_{n} \cdot x)=
\frac{||\bigwedge^2 R_n(X_{n+1} \cdots X_{2n} v \wedge v)||}{||R_{2n} v||\,||R_n v||}\leq \frac{||\bigwedge^2 R_n||\, ||v||^2\,||X_{n+1} \cdots X_{2n}||}
{||R_{2n} v||\,||R_n v||}.
\end{equation*}
Now, we apply Proposition \ref{largedeviations1}   to $||\bigwedge^2 R_n||$, $||R_{2n}v||$, $||R_n v||$ and to  $||X_{n+1}\cdots X_{2n}||$ (which has the same distribution as $||R_n||$). Since $\lambda_1(\mu)>\lambda_2(\mu)$, \eqref{eq1} follows along the same lines as in the proof of Lemma \ref{lemlem1}.  
\end{proof}

The following lemma expresses the fact that, with high probability, the attracting directions of random products do not stay exponentially close to the projective subspace of lower expansion.

\begin{lemma}\label{while} Under the assumptions of Proposition \ref{prop.distance}, for every $\epsilon>0$, there exists a summable sequence $(D_n)_{n\in \N}$ such that for every $n\in \N$, 
$$\p\left(\delta(x_{L_n}^+, [L_{\mu}])\leq e^{-\epsilon n}\right)<D_n.$$
\end{lemma}

\begin{proof}
Fix $x=[v] \in \PV \setminus [L_\mu]$. 
By Lemma \ref{lemlem1}, there exist a constant $C>0$ and a summable sequence $D'_n$ of positive reals such that for every $n \geq 1$,  
\begin{equation}\label{eq.6}
\mathbb{P}(\delta(x^+_{L_n},L_n \cdot x) \geq e^{-Cn}) <D'_n. 
\end{equation}
Let now $\epsilon>0$. By \eqref{eq.6}, 

\begin{equation}\label{eq.66}\p\left(\delta(x_{L_n}^+, [L_{\mu}])\leq e^{-\epsilon n}\right) < \p\left(\delta(L_n \cdot x, L_{\mu}) \leq   e^{-\epsilon  n} +e^{-C n}\right) +D'_n.\end{equation}
As in the proof of Proposition \ref{largedeviations1}, for every $v\in V\setminus L_{\mu}$, denote by $\overline{v}$ its projection onto the quotient $V/L_\mu$. Endow $V/L_{\mu}$ with the norm $||v+L_{\mu}||:=\inf\{||v+w||; w\in L_{\mu}\}$ and denote by $\pi: G\longrightarrow \GL(V/L_{\mu})$ the canonical projection. By $(\ref{eq.norms})$, we have 
$$
\delta(L_n \cdot x, [L_{\mu}])=\frac{||\pi(L_n)\bar{v}||}{||L_n v||}
.$$
 Recalling that $\pi_\ast \mu$ has a unique cocycle average, an application of Proposition \ref{largedeviations1} to $||\pi(L_n)\bar{v}||$ and $||L_n v||$ implies that  there exists a summable sequence $D''_n(\epsilon)$ of positive reals such that for every $n \geq 1$,  
\begin{equation}\label{eq.7}
\mathbb{P}(\delta(L_n \cdot x, [L_\mu]) \leq   e^{-\epsilon  n} +e^{-C n})\leq D_n''(\epsilon).
\end{equation}
Combining \eqref{eq.66} and  \eqref{eq.7} proves the lemma for  $D_n(\epsilon):=D'_n+D''_n(\epsilon)
$.  
\end{proof}

In the next lemma, we estimate the probability of return of the attractive point of $L_n$ to exponentially small neighborhoods of hyperplanes. 
   \begin{lemma}\label{lemlem4}
Under the assumptions of Proposition \ref{prop.distance},  
  for any $\epsilon>0$, there exists  a summable sequence  $(D_n)_{n\in \N}$ of positive reals   
   such that for every $n\in \N$,
$$\underset{H \in \mathcal{H}_{n,\epsilon}}{\sup}\p\left(\delta(x_{L_n}^+,H)\leq e^{-\epsilon n}\right)<D_n,$$
 where $\mathcal{H}_{n,\epsilon}$ is the set of projective hyperplanes $H=(\C u)^{\perp}$ such that $\delta([u], [L_{\check{\mu}}])>e^{-\epsilon n/2}$. 
 \end{lemma}
 
\begin{proof}
We start by noting the following elementary inequality (see e.g.\  \cite[Lemma 4.1]{BFLM}): for every $g\in \GL(V)$ and every $u\in V\setminus \{0\}$, 
\begin{equation}\frac{||g^\ast u||}{||g^\ast||\,||u||}\leq \delta(x_g^+, (\C u)^{\perp}) +\frac{a_2(g)}{a_1(g)}.\label{geometry1}
\end{equation}
Since $\lambda_1(\mu)>\lambda_2(\mu)$,  by $(\ref{nulaij})$, there exist $C>0$  and a summable sequence $(D'_n)_n$ such that the following holds for every $n$: 
 \begin{equation}\label{geometry3}\p(a_2(L_n)/a_1(L_n)<e^{-Cn})>1-D'_n.\end{equation} 
 Let now $\epsilon>0$. 
Applying Proposition  \ref{largedeviations1} to the probability $\check{\mu}$, we deduce that there exists a summable sequence $D''_n(\epsilon)$ such that for every non-zero vector $u$ of $V$ and for every $n \in \mathbb{N}$, 
\begin{equation}\label{geometry2}\p\left(\frac{||L_n^\ast u||}{||L_n^\ast||\,||u||}\leq e^{-n \epsilon/4} \delta([u], [L_{\check{\mu}}])\right)<D''_n(\epsilon).\end{equation}
Let $D_n(\epsilon)=D'_n+D''_n(\epsilon)$. Combining 
 \eqref{geometry1}, \eqref{geometry3} and  \eqref{geometry2},  we deduce that   for every $H=(\C u)^\perp$ and for every $n$: 
$$\p\left(\delta(x_{L_n}^+, H)\leq e^{-n\epsilon/4} \delta([u],[L_{\check{\mu}}]) -e^{-C n}\right)< D_n(\epsilon).$$
Thus, for every $n$, 
$$\sup_{H\in \mathcal{H}_{n,\epsilon}}\p\left(\delta(x_{L_n}^+, H)\leq e^{-3 n\epsilon/4}  -e^{-C n}\right)< D_n(\epsilon),$$
with $\mathcal{H}_{n,\epsilon}$ defined as in the statement. 
Therefore, the lemma is proved for all $0<\epsilon<C$; a fortiori for every $\epsilon>0$.
\end{proof}

\begin{remark}
The content of the previous lemma is closely related to regularity properties of stationary measures on the projective space. It is shown in \cite[Theorem 2.4]{Aoun-Guivarch} that, when $\lambda_1(\mu)>\lambda_2(\mu)$, there exists a unique $\mu$-stationary probability measure $\nu$ on the open subset $\PV\setminus [L_{\mu}]$ of $\PV$. The measure $\nu$ is the limit distribution of $x_{L_n}^+$ and the projective subspace generated by the support of $\nu$ is $U_{\mu}:=[L_{\check{\mu}}^{\perp}]$. This explains the condition on $H$ appearing in the previous lemma. Finally, we remark that since $\nu$ is non-degenerate on $U_{\mu}$ (i.e.\ does not charge any projective hyperplane), the previous lemma is an additional quantitative  information on the regularity of $\nu$.
\end{remark}

\begin{proof}[Proof of Proposition \ref{prop.distance}]
 Since $L_n^\ast=X_1^\ast \cdots X_n^\ast$ has the same distribution as the right random walk for the probability measure $\check{\mu}$, and since $\check{\mu}$ and $\mu$ have the same Lyapunov exponents, Lemma \ref{lemlem2} gives some $C>0$  and a summable sequence $(D_n)_n$ such that  for every $n$,   
\begin{equation}\label{y3}\p\left(\delta(x_{L_n^\ast}^+,
 x_{L_{\lfloor n/2 \rfloor}^\ast}^+)\geq e^{-Cn}\right)\leq D_n.\end{equation}
Let  $M_n:=X_n \cdots X_{\lfloor n/2 \rfloor +1}$. Since the $n$-tuples $(X_1, \ldots, X_n)$ and $(X_n, \ldots, X_1)$ have the same distribution,  the same lemma gives that
\begin{equation}\label{y4}\p\left( \delta(x_{L_n}^+,x_{M_n}^+) \geq   e^{-Cn}\right)=  \p\left( \delta(x_{R_n}^+,x_{R_{n-\lfloor n/2 \rfloor}}^+) \geq   e^{-Cn} \right)\leq D_n.\end{equation}
\noindent Fix now $0<\epsilon<C$. Since for every $n \in \mathbb{N}$,  $H_{L_n}^{<}=(x_{L_n^\ast}^+)^{\perp}$, by \eqref{y3} and \eqref{y4}, we get that for every $n$, 
\begin{equation}\label{zaatar1} 
\p\left(\delta(x_{L_n}^+, H_{L_n}^<) \leq e^{-\epsilon n} \right)  \leq   2 D_n + \p\left( \delta\left(x_{M_n}^+, H_{L_{\lfloor n/2 \rfloor}}^< \right)\leq 3 e^{-\epsilon n} \right). 
\end{equation}
But by Lemma  \ref{while}, we can assume that $\p\left(\delta(x_{L_n^\ast}^+, [L_{\check{\mu}}])>e^{- \epsilon n/2}\right)>1-D_n$. Hence, since  
   $M_n$ and $L_{\lfloor n/2 \rfloor}$ are independent random variables for every $n$,   \eqref{zaatar1} yields that: 
$$\p\left(\delta(x_{L_n}^+, H_{L_n}^<) \leq e^{-\epsilon n} \right)\leq 3D_n+ \underset{H\in \mathcal{H}_{n,\epsilon}}{\sup} {\p\left(\delta(x_{M_n}^+, H)\leq 3e^{- \epsilon n} \right)},$$
where $\mathcal{H}_{n,\epsilon}$ is the set of projective hyperplanes $H=(\C u)^\perp$ such that $\delta([u], [L_{\check{\mu}}])>e^{-n\epsilon/2}$. 
Since $M_n$ has same distribution as $R_{n-\lfloor n/2 \rfloor}$, we conclude by Lemma \ref{lemlem4}. 
\end{proof}

The following is a qualitative version of Proposition \ref{prop.distance} under the weaker finite first moment assumption.

\begin{proposition}\label{prop.distance.1}
Assume $\mu$ has  finite first moment and that $\lambda_1(\mu)>\lambda_2(\mu)$. Then for every $\epsilon>0$, 

$$\lim_{n\to +\infty}{\p\left(\delta(x_{L_n}^+, H_{L_n}^<)\leq e^{-\epsilon n}\right)} = 0.$$
\end{proposition}

To avoid repetition, we limit ourselves to indicate the proof of this result: 

\begin{proof}
One readily checks that using Proposition \ref{large.deviations.2}  (instead of Proposition \ref{largedeviations1}) in the proofs of Lemmas \ref{lemlem1}, \ref{lemlem2}, \ref{while} and \ref{lemlem4}, we obtain the conclusions of these lemmas with $D_n\underset{n\rightarrow +\infty}{\longrightarrow} 0$ instead of $\sum_n{D_n}<+\infty$. Now the proof of Proposition \ref{prop.distance} applies verbatim to yield the desired result.  
\end{proof}

\section{End of the proof}\label{end.proof}
The strategy of the proof is as follows. Using the estimates of  the previous section, we first prove the theorems when  $\lambda_1(\mu)>\lambda_2(\mu)$. We then  deduce the result in the case of a Lyapunov gap.  Finally we check the validity of the result when there is no Lyapunov gap.
 
\begin{proof}[Proof of Theorems \ref{main} and \ref{main1}] 

We first prove the strong law. 

\begin{enumerate}
\item Suppose $\lambda_1(\mu)>\lambda_2(\mu)$. Combining Lemma \ref{geometry}, Proposition \ref{prop.distance} and $(\ref{nulaij})$, we get that for every $\epsilon>0$, $$\sum_{n\in \N}{\p\left(\frac{\rho(L_n)}{||L_n||} \leq e^{-\epsilon n}\right)}<+\infty.$$ 
By Borel--Cantelli's lemma, we deduce that for every $\epsilon>0$, almost surely, there exists a random integer $n_0(\epsilon)$ such that for every $n\geq n_0(\epsilon)$,  $-\epsilon \leq \frac{1}{n} \log \frac{\rho(L_n)}{||L_n||}\leq 0$.
We deduce that the sequence $\frac{1}{n} \log \frac{\rho(L_n)}{||L_n||}$ converges to $0$ almost surely. We conclude by Furstenberg-Kesten's theorem. 
\item Assume now that there exists $s=2, \dots, d-1$ such that $\lambda_1(\mu)=\lambda_2(\mu)=\cdots = \lambda_s(\mu)>\lambda_{s+1}(\mu)$.  Denote by $\eta$ the probability measure  on $GL(\bigwedge^s V)$, given by the pushforward of $\mu$ by the group homomorphism $g\mapsto \bigwedge^s g$. Clearly,  $\eta$ has finite second moment and  $\lambda_1(\eta)>\lambda_2(\eta)$. Applying  Theorem \ref{main1} to $\eta$, we get that   almost surely $\frac{1}{n} \log \rho(\bigwedge^ s L_n) \underset{n\rightarrow +\infty}{\longrightarrow} \lambda_1(\eta)=s \lambda_1(\mu)$. But, for every $n\in \N$,  $\rho(\bigwedge^s L_n)\leq \rho(L_n)^s$. Thus  $\lambda_1(\mu)\leq \underset{n\rightarrow +\infty}{\liminf}{\frac{1}{n} \log \rho(L_n)}$. \\
Since $\underset{n\rightarrow +\infty}{\limsup} {\frac{1}{n} \log \rho(L_n)}\leq \lambda_1(\mu)$, we are done. 
\item 

 Finally, if  $\lambda_1(\mu)=\cdots = \lambda_d(\mu)$,    the desired convergence follows   from the  inequality   $a_d(g)\leq \rho(g)\leq a_1(g)$ true for every $g\in \GL_d(\C)$. 
 
\end{enumerate}
This proves Theorem \ref{main}. By replacing the use of Proposition \ref{prop.distance} by that of Proposition \ref{prop.distance.1} in the proof of Theorem \ref{main}, we deduce the convergence in probability of  the sequence  $\frac{1}{n} \log{\rho(L_n)}$ to $\lambda_1(\mu)$. To conclude the convergence in $\textrm{L}^1$, it is enough to show that the family $\{\frac{1}{n} \log{\rho(L_n)} \, | \, n\in \N\}$ is uniformly integrable. Indeed, for every $g\in \GL_d(\C)$, $\frac{1}{N(g)}\leq \rho(g)\leq N(g)$, where   $N(g)=\max\{||g||, ||g^{-1}||\}$. 
By submultiplicativity  of $N(\cdot)$, this implies that $|\frac{1}{n} \log \rho(L_n)|\leq \frac{1}{n}\sum_{i=1}^n{\log N(X_i)}$ for every $n\in \N$. 
Since the random variables $N(X_i)$ are i.i.d. and integrable, we deduce  from  the strong law of large numbers that the family $\{\frac{1}{n} \log{\rho(L_n)} \, | \, n\in \N\}$ is uniformly integrable. The proof is done. 
\end{proof}

We end this paper by giving an example of an ergodic stationary random walk $L_n=X_n \cdots X_1$, with increments $X_i$  in Markovian dependence, and such that the spectral radius $\rho(L_n)$ fails to satisfy a law of large numbers. This contrasts with the law of large numbers for the norm of $L_n$.\\

\begin{example}\label{markov}
Consider the following $3\times 3$ matrices in $\SL_3(\R)$: 
$$a:= \left( \begin{matrix} 
3 & 0 & 0\\
0 & 1 & 0\\
0 & 0 & 1/3 
\end{matrix}\right) \quad , \quad 
\sigma:=\left(\begin{matrix} 0 & 0 & 1\\
1 & 0 & 0\\
0 & 1 & 0\end{matrix}\right) \quad \text{and} \quad     \omega:=\left(\begin{matrix} 0 & 1 & 0\\
0 & 0 & 1\\
1 & 0 & 0\end{matrix}\right).$$
Let $(X_n)_n$ denote the Markov chain on the finite state $E=\{a,\sigma,\omega\}$ whose transition probabilities are given by
$P(a,a)=1/2$, $P(a,\sigma)=1/2$, $P(\sigma, \omega)=1$, $P(\omega,a)=1$ and all the other transition probabilities are zero. 
This is an irreducible and aperiodic Markov chain with unique stationary probability measure $m:=\frac{1}{2}\delta_{a} + \frac{1}{4}\delta_{\sigma}+\frac{1}{4}\delta_{\omega}$. From now, we assume that $X_1$ has initial distribution $m$, so that $(X_i)_i$ is an ergodic stationary process. 

As usual, for every $n \in \mathbb{N}$, let $L_n=X_n\cdots X_1$. We claim that, almost surely, 
\begin{equation}\label{eq.claim}
 \underset{n\rightarrow +\infty}{\liminf}\frac{1}{n}\log \rho(L_n)=0 <  \underset{n\rightarrow +\infty}{\limsup}{\frac{1}{n}}\log \rho(L_n)=\lambda_1.
\end{equation}
To see this, denote by  $\tau^{(n)}$ the $n^{\textrm{th}}$ return time of the Markov Chain  $(X_n)_{n\geq 1}$ to $a$. It follows by construction of the transition kernel that we have  $\tau^{(n)}\leq 3n+1$ for every $n \geq 1$. Moreover, conditionally on the event $\{X_1=a\}$, we have  $L_{\tau^{(n)}}=a^{n}$. 
As a consequence, one easily deduces that almost surely
$$\underset{n\rightarrow+\infty}{\limsup}
\frac{1}{n} \log\rho(L_n) \geq \frac{1}{3} \log 3>0.$$

Let now $G$ be the group generated by $a,\sigma,\omega$ and $H\subset G$ be the subgroup consisting of  diagonal matrices  so that $G$ is the union of cosets $\tau H$ for $\tau \in \{1, \sigma,\omega\}$.
Each $g\in \sigma H \cup \omega H$ is a generalized permutation matrix with spectral radius equal to one. 
On the other hand, by construction, we have 
$\p\left( \limsup \{L_n \not\in H \} \right)=1$, i.e.\ almost surely $L_n\not\in H$ infinitely often. This shows that a.s.  $$\underset{n\rightarrow+\infty}{\liminf}
\frac{1}{n} \log\rho(L_n)=0,$$ 
and concludes the proof of \eqref{eq.claim}.
\end{example}

\begin{remark}(Further directions)  
Note that the underlying algebraic structure of the Markov chain in the previous example is somewhat degenerate. For instance the group generated by the state space is irreducible but not strongly irreducible. Moreover, the  ergodic stationary Markov measure $\mathbb{P}_m$ on the shift space is not of full-support. These lead to the natural problem of identifying the probabilistic and algebraic assumptions under which the spectral radius of a random walk with ergodic stationary increments satisfies a law of large numbers.
\end{remark}


\end{document}